\documentclass[12pt]{article}
\usepackage{}
\usepackage[numbers,sort&compress]{natbib}

\usepackage{color}
\usepackage{threeparttable}

\usepackage[english]{babel}
\usepackage{diagbox}
\usepackage{t1enc}
\usepackage{amsthm,amsmath,amssymb}
\usepackage{mathrsfs}
\usepackage[mathscr]{eucal}
\usepackage[latin1]{inputenc}
\usepackage[english]{babel}
\usepackage{latexsym}
\usepackage{graphicx}
\usepackage[noend]{algpseudocode}
\usepackage{algorithmicx,algorithm}
\usepackage{lineno}
\usepackage{booktabs}
\usepackage{multirow}
\newtheorem{theorem}{Theorem}
\newtheorem{corollary}{Corollary}
\newtheorem{proposition}{Proposition}
\newtheorem{lemma}{Lemma}
\newtheorem{remark}{Remark}
\newtheorem{conjecture}{Conjecture}
\newtheorem{example}{Example}

\newtheorem{definition}{Definition}

\def \RO{\textup{RO}}
\def\sfp{\textup{sfp}}
\def\sqrt{\textup{sqrt}}
\def \T{\textup{T}}

\def \rank{\textup{rank}}
\def \rank{\textup{rank}}

\def \diag{\textup{diag~}}

\def \lcm{\textup{lcm}}
\def \Res {\textup{\Res}}
\def \ord{\textup{ord}}
\makeatletter
\newcommand{\rmnum}[1]{\romannumeral #1}
\newcommand\restr[2]{{
		\left.\kern-\nulldelimiterspace 
		#1 
		\right|_{#2} 
}}
\newcommand{\Rmnum}[1]{\expandafter\@slowromancap\romannumeral #1@}
 
\makeatother
\title{Haemers' conjecture: an algorithmic perspective}
\author{\small Wei Wang$^{{\rm a}}$\quad\quad Wei Wang$^{\rm b}$\thanks{Corresponding author: wang\_weiw@163.com}
\\
{\footnotesize$^{\rm a}$School of Mathematics, Physics and Finance, Anhui Polytechnic University, Wuhu 241000, P. R. China}\\
{\footnotesize$^{\rm b}$School of Mathematics and Statistics, Xi'an Jiaotong University, Xi'an 710049, P. R. China}
}
\date{}
\begin{document}
 \maketitle

\begin{abstract}
Characterizing graphs by their spectra is a fundamental and challenging problem in spectral graph theory, which has received considerable attention in recent years. A major unsolved conjecture in this area is Haemers' conjecture which states that almost all graphs are determined by their spectra. Despite many efforts, little is known about this conjecture so far. In this paper, we shall consider Haemers' conjecture from an algorithmic perspective. Based on some recent developments in the generalized spectral characterizations of graphs, we propose an algorithm to find all possible generalized cospectral mates for a given $n$-vertex graph $G$, assuming that $G$ is controllable or almost controllable. The experimental results indicate that the proposed algorithm runs surprisingly fast for most graphs with several dozen vertices. Moreover, we observe in the experiment that most graphs are determined by their generalized spectra, e.g., at least 9945 graphs are determined by their generalized spectra among all randomly generated 10,000 graphs on 50 vertices in one experiment. These experimental results give strong evidence for Haemers' conjecture. \\

\noindent\textbf{Keywords}: generalized spectrum; generalized spectral characterization; generalized cospectral mates; Smith normal form; controllable graph

\noindent
\textbf{AMS Classification}: 05C50
\end{abstract}
\section{Introduction}
\label{intro}
A fundamental and challenging problem in spectral graph theory is to determine whether a given graph is determined by its spectrum. This problem has received considerable attention; see the survey papers \cite{ervdamLAA2003,ervdamDM2009}.

A major unsolved conjecture in this area is the following conjecture of Haemers.
\begin{conjecture}[\cite{ervdamLAA2003,haemers2016}]
	Almost all graphs are determined by their spectra.
\end{conjecture}
For simplicity, for a graph $G$,  we say that $G$ is a DS graph if $G$ is determined by its spectrum. To have an intuitive feeling on Haemers' conjecture, one may want to know the exact percentages of DS graphs for a fixed order $n$.  This has be accomplished only for $n\le 12$ by enumeration method \cite{godsil1976, haemers2004,brouwer2009}, see Table \ref{computer}.
\begin{table}[htbp]
	\footnotesize
	\centering
	\caption{\label{computer} Fractions of DS and DGS graphs}
	\begin{threeparttable}
		\begin{tabular}{crrcrcc}
			\toprule
			$n$ &\# graphs & \# DS graphs &  DS graphs(\%)& \# DGS graphs &  DGS graphs(\%)\\
			\midrule
			1  &  1 &1&100&1&100\\
			2  &  2 &2&100&2&100\\
			3  & 4 &4&100&4&100\\
			4  & 11 &11&100&11&100\\
			5  &  34 &32&94.1&34&100\\
			6  &  156 &146&93.6&156&100\\
			7  &  1044 &934&89.5&1004&96.2\\
			8  &  12346 &10624&86.1&11180&90.6\\
			9  &  274668 &223629&81.4&230857&84.0\\
			10  &  12005168 &9444562&78.7&9587016&79.9\\
			11  &  1018997864 &803666188&78.9&806733492&79.2\\
			12  &  165091172592 &134023600111&81.2&134135405065&81.2\\
			13  &  50502031367952 &?&?&?&?\\
			\bottomrule
		\end{tabular}
	\end{threeparttable}
\end{table}

In this paper, we shall consider a weaker version of Haemers' conjecture which states that almost all graphs are determined by their generalized spectra \cite{abiad2012}. Here, the generalized spectrum of a graph $G$ refers to the spectrum of $G$ together with that of $\overline{G}$, the complement of $G$. We shall say that $G$ is DGS (or $G$ is a DGS graph) if $G$ is determined by its generalized spectrum.
 Recently, some (sufficient) conditions for a graph to be DGS were obtained \cite{wang2013ElJC,wang2017JCTB,wang2021EUJC,qiuDM2022}. Using these conditions and the Monte Carlo method for graphs with several dozen vertices, the author \cite{wang2017JCTB} observed that nearly $20\%$ of graphs are DGS.  Unfortunately, about $80\%$ graphs do not satisfy these conditions and so we do not know whether they are DGS or not.

 The current paper is a continuation of this line of research. However, instead of seeking sufficient conditions for a graph to be DGS, we try to obtain a necessary and sufficient condition, which should be valid for almost all graphs; this was done algorithmically.  For a graph $G$, we try to \emph{generate} all possible generalized cospectral mates of $G$.  Here a \emph{generalized cospectral mate} of a graph
 $G$ means a graph that is generalized cospectral with  $G$ but is not isomorphic to $G$. Note that a DGS graph is just a graph with no generalized cospectral mates. Clearly, if we can generate all generalized cospectral mates for a given graph in an efficient way, then we can quickly determine whether the given graph is DGS or not. Although the proposed algorithm does not work for all graphs, it does work on most graphs, as we shall see later. We recall some definitions for an explanation.

 For a graph $G$ with $n$ vertices, the \emph{walk matrix} of $G$ is
 $$W(G)=[e,Ae,\ldots,A^{n-1}e],$$
 where $e$ is the $n$-dimensional all-ones vector and $A$ is the adjacency matrix of $G$. A graph $G$ is \emph{controllable} \cite{godsil2012} (resp. \emph{almost controllable} \cite{wang2021EUJC}) if $\rank\, W(G)$, the rank of $W(G)$, equals $n$ (resp. $n-1$).  O'Rourke and Touri \cite{rourke2016} proved that almost all graphs are controllable. In our algorithm, we shall include the class of almost controllable graphs. Indeed, the analysis for controllable graphs can be easily extended to almost controllable graphs.


 The controllability or almost controllability assumption is crucial for our discussion.  An orthogonal matrix $Q$ is \emph{regular} if $Qe=e$; and is \emph{rational} if each entry of $Q$ is rational. The starting point of our algorithm is the fundamental connection between generalized cospectrality and rational regular orthogonal matrices. This connection dates back to an old result of Johnson and Newman \cite{johnson1980JCTB}, whereas the rationality of the associated matrix $Q$ and the consequential applications in generalized spectral characterizations of graphs are recognized in \cite{wang2006EuJC} for controllable graphs, and in \cite{wang2021EUJC} for almost controllable graphs.

 \begin{theorem}[\cite{johnson1980JCTB,wang2006EuJC,wang2021EUJC}]\label{connection}
 	Let $G$ and $H$ be two graphs with $n$ vertices. Then $G$ and $H$ are generalized cospectral if and only if there exists a regular orthogonal matrix $Q$ such that
 	\begin{equation}\label{qaq}
 Q^\T A(G)Q=A(H).
 	\end{equation} Moreover,
 	
 	\textup{(\rmnum{1})} if $G$ is controllable then $Q^\T=W(H)(W(G))^{-1}$ and hence $Q$ is unique and rational.
 	
 	\textup{(\rmnum{2})} if $G$ is almost controllable then  Eq.~\eqref{qaq} has exactly two solutions for $Q$, both of which are rational.
 \end{theorem}
From Theorem \ref{connection}, in order to find all generalized cospectral mates of a graph $G$, it suffices to find all regular orthogonal matrices $Q$ such that $Q^\T A(G) Q$ is the adjacency matrix for some graph. Clearly, the assumption that $\rank~W(G)\ge n-1$ simplifies the problem significantly since we may require further the desired $Q$'s to be rational.

The heart of our algorithm is to give a \emph{practical} way to generate all $Q$'s such that  $Q^\T A(G) Q$ is the adjacency matrix for some graph. In order to give an overview of the algorithm, we need to  recall a simple but crucial concept.
\begin{definition}[\cite{wang2006EuJC}]\textup{
		Let $Q$ be a rational matrix. The \emph{level} of $Q$, denoted by $\ell(Q)$, is the smallest positive integer $k$ such that $kQ$ is an integral matrix.}
\end{definition}
For simplicity, assume for the moment that $G$ is controllable. Then it is easy to see from Theorem \ref{connection}(\rmnum{1}) that $\ell(Q)|\det W(G)$, which gives a rough estimate for the levels of all possible $Q$'s. Write $L=\det W(G)$ and let $x$ be any column of $LQ$ for all potential $Q$'s. Clearly $x$ is an integral vector. It is not difficult to see that $x$ must satisfy the following fundamental congruence equation:
\begin{equation}\label{solx}
(W(G))^\T x\equiv 0\pmod {L}
\end{equation}
The main part of our algorithm is to find all integral solutions to Eq.~\eqref{solx} which also satisfy some evident constraints such as $x^\T x=L^2$.  In order to find all potential $Q$'s, we shall use these solutions as vertices to  construct a new graph $\Omega(G)$ with the property that  every potential $Q$ naturally corresponds to a maximum clique of  $\Omega(G)$.

However, it should be pointed out that $\det W(G)$ is usually huge even for graphs with only several dozen vertices. So a naive approach does not work. In Section 2, we review some recent results to bound the maximal levels of possible $Q$'s using some invariants of $G$ under generalized cospectrality. The derivation is  based on what we refer to as \emph{modified walk matrices} which results in some refinement of Eq. \eqref{solx}.  In Section 3, we give the details to solve the associated equations and to construct the new graph $\Omega(G)$.  This leads to a practical way to generate all generalized cospectral mates for a graph. In particular, we give a necessary and sufficient condition (in an algorithmic way) for a graph $G$ to be DGS, under the  assumption that $G$ is controllable or almost controllable. As an important application, we use the proposed algorithm to test Haemers' conjecture  in Section \ref{sim}. The new experiment results give strong evidence for this conjecture, see Table \ref{res}.

\section{Preliminaries}
For an integer $n\times n$ matrix $M$, there always exist two unimodular matrices $U$ and $V$ such that $UMV$ is a diagonal matrix
$\diag[d_1,d_2,\ldots,d_r,0,\ldots,0]$, where $r$ is the rank of $M$ and the first $r$  diagonal entries are positive integers satisfying  $d_1\mid d_2\mid\cdots\mid d_r$. The diagonal entries $d_1,d_2,\ldots,d_r,0,\ldots,0$ are uniquely determined by the given matrix $M$. The unique  diagonal matrix
$\diag[d_1,d_2,\ldots,d_r,0,\ldots,0]$ is  known as the \emph{Smith normal form} of $M$,  and the diagonal entries  are \emph{invariant factors} of $M$.

We use $\RO_n$ to denote the group of regular orthogonal matrices of order $n$ (under the usual matrix multiplication).
For a graph $G$ with $n$ vertices, define
$$\mathcal{Q}(G)=\{Q\in \RO_n\colon\,Q^\T A(G)Q=A(H)\text{~for some graph~}H\}.$$
We note that $\mathcal{Q}(G)$ always contains all permutation matrices of order $n$. Moreover, if $Q\in \mathcal{Q}(G)$ then $QP\in \mathcal{Q}(G)$ for any permutation matrix $P$. Indeed, $(QP)^\T A(G) (QP)=P^\T A(H) P$, corresponding to a relabeling of $H$.
\subsection{Basic estimate of levels}

We know by Theorem \ref{connection} (\rmnum{1}) and (\rmnum{2}) that all matrices in $\mathcal{Q}(G)$ are rational provided that $G$ is either controllable or almost controllable. For the case of controllable graphs, the walk matrix $W(G)$ can be used to restrict all possible levels for $Q\in \mathcal{Q}(G)$.
For an integral $n\times n$ matrix $M$ and a positive integer $k\in \{1,2,\ldots,n\}$, we use $d_k(M)$ to denote the $k$-th invariant factor of $M$.
\begin{lemma}[\cite{wang2006EuJC}]\label{lqc}
If $G$ is controllable, then $\ell(Q)\mid d_n(W(G))$ for any $Q\in \mathcal{Q}(G)$.
\end{lemma}
\begin{proof}
	Let $U$ and $V$ be unimodular matrices such that $UW(G)V$ is $\diag[d_1,d_2,\ldots,d_n]$, the Smith normal form of $W(G)$. By Theorem \ref{connection} (\rmnum{1}), we have
	$$d_n Q^\T=d_nW(H)(W(G))^{-1}=W(H)V\diag\left[\frac{d_n}{d_1},\ldots,\frac{d_n}{d_{n-1}},1\right]U,$$
	indicating that $d_n Q^\T$ is an integral matrix. Thus, $\ell(Q)\mid d_n$, as desired.
	\end{proof}
Lemma \ref{lqc} was extended to almost controllable graphs in \cite{wang2021EUJC} using a modified walk matrix.
\begin{definition}[\cite{wang2021EUJC}]\label{defxi}
	Let $G$ be an almost controllable graph of order $n$. Write
	\begin{equation*}
	\xi(G)=\left(\begin{matrix}
	W_{1n}(G)\\
	W_{2n}(G)\\
	\vdots\\
	W_{nn}(G)\\
	\end{matrix}
	\right) \text{~and~} W_\delta(G)=[e,A(G)e,\ldots,A^{n-2}(G)e,(-1)^\delta\frac{1}{2^{\lfloor\frac{n}{2}\rfloor-1}}\cdot \xi(G)], \delta\in\{0,1\},
	\end{equation*}
	where $W_{i,n}(G)$ denotes the algebraic cofactor of the $(i,n)$-entry for $W(G)$.
\end{definition}
\begin{lemma}[\cite{wang2021EUJC}]\label{lqa}
If $G$ is almost controllable, then $\ell(Q)\mid d_n(W_0(G))$ for any $Q\in \mathcal{Q}(G)$.
\end{lemma}
\begin{proof}
Let $H$ be the  graph such that $A(H)=Q^\T A(G)Q$. Then $W(H)=Q^\T W(G)$ and hence $\rank~W(H)=\rank~W(G)=n-1$, that is, $H$ is almost controllable. It was shown in \cite{wang2021EUJC} that $Q^\T=W_0(H)(W_0(G))^{-1}$ or $Q^\T=W_0(H)(W_1(G))^{-1}$. Noting that  $W_0(G)$ and $W_1(G)$ have the same Smith normal form, the lemma follows by the same argument as in Lemma \ref{lqc}.
\end{proof}
\begin{definition}\normalfont{
	Let $f(x)=x^n+a_1x^{n-1}+a_2x^{n-2}+\cdots+a_n$ be a monic polynomial with integral matrix. The \emph{discriminant} of $f(x)$ is
	\begin{equation}
	\Delta(f)=\prod_{i>j}(\alpha_i-\alpha_j)^2,
	\end{equation}
	where $\alpha_1,\alpha_2,\ldots,\alpha_n$ are roots of $f(x)$ over $\mathbb{C}$. For a graph $G$, the \emph{discriminant} of $G$, denoted by $\Delta(G)$ is defined to be the discriminant of its characteristic polynomial, i.e.,  $\Delta(G)=\Delta(\det(xI-A(G)))$. }
\end{definition}
It is known that $\Delta(f)=\pm\textup{Res}(f,f')$, where $\textup{Res}(f,f')$ is the resultant of $f$ and its derivative $f^{'}$, see, e.g., \cite{lang2002}. This means that $\Delta(G)$ is always an integer.
\begin{lemma}[\cite{wang2006arxiv}]\label{lqd}
	Let $G$ be a graph and $Q$ be any rational matrix in $\mathcal{Q}(G)$. If $p$ is an odd prime factor of $\ell(Q)$ then $p^2\mid \Delta(G)$.
\end{lemma}
For a nonzero integer $m$ and a prime $p$, we use $\ord_p(m)$ to denote the maximum nonnegative integer $k$ such that $p^k$ divides $m$. For a set $S$ of integers, we use $\lcm \,S$ to denote the least common multiple of integers in $S$.
\begin{definition}
	Let $G$ be  controllable or almost controllable.	Define \begin{equation}
	\ell(G)=\lcm\{\ell(Q)\colon\,Q\in \mathcal{Q}(G)\}.
	\end{equation}
\end{definition}
 We summarize Lemmas \ref{lqc}-\ref{lqd} as the following proposition, which  we shall refer to as the basic estimate of $\ell(G)$.
\begin{proposition}\label{lG1}
	Let $G$ be controllable or almost controllable. Write
	\begin{equation}
	d=\begin{cases}
	d_n(W(G))&\text{if~} G\text{~is controllable,}\\
	d_n(W_0(G))&\text{if~} G\text{~is almost controllable.}
	\end{cases}
	\end{equation}
	Then we have\\	
	\textup{(\rmnum{1})} $\ord_2(\ell(G))\le \ord_2(d)$;\\	
	\textup{(\rmnum{2})} for odd prime $p$, $\ord_p(\ell(G))\le \ord_p(d)$ if $p\mid\gcd(\Delta(G),d)$ and $p^2\mid\Delta(G)$; otherwise $\ord_p(\ell(G))=0$.
\end{proposition}

\begin{example}\label{g13}\normalfont{
	Let $G$ be a graph with $n=13$ vertices whose adjacency matrix is
	
\begin{equation}A=\small{\left(
\begin{array}{ccccccccccccc}
	0 & 1 & 0 & 0 & 1 & 1 & 0 & 0 & 0 & 0 & 1 & 0 & 0 \\
1 & 0 & 1 & 1 & 0 & 0 & 0 & 0 & 0 & 1 & 1 & 1 & 1 \\
0 & 1 & 0 & 1 & 1 & 0 & 0 & 1 & 1 & 0 & 1 & 0 & 0 \\
0 & 1 & 1 & 0 & 1 & 0 & 1 & 0 & 0 & 0 & 1 & 0 & 0 \\
1 & 0 & 1 & 1 & 0 & 0 & 1 & 0 & 0 & 0 & 1 & 0 & 0 \\
1 & 0 & 0 & 0 & 0 & 0 & 0 & 1 & 1 & 1 & 0 & 0 & 1 \\
0 & 0 & 0 & 1 & 1 & 0 & 0 & 0 & 0 & 0 & 1 & 0 & 0 \\
0 & 0 & 1 & 0 & 0 & 1 & 0 & 0 & 0 & 1 & 0 & 0 & 0 \\
0 & 0 & 1 & 0 & 0 & 1 & 0 & 0 & 0 & 0 & 0 & 0 & 0 \\
0 & 1 & 0 & 0 & 0 & 1 & 0 & 1 & 0 & 0 & 1 & 0 & 0 \\
1 & 1 & 1 & 1 & 1 & 0 & 1 & 0 & 0 & 1 & 0 & 0 & 1 \\
0 & 1 & 0 & 0 & 0 & 0 & 0 & 0 & 0 & 0 & 0 & 0 & 1 \\
0 & 1 & 0 & 0 & 0 & 1 & 0 & 0 & 0 & 0 & 1 & 1 & 0 \\
\end{array}
\right).}
\end{equation}
}
\end{example}
Using Mathematica, we find that the Smith normal form of $W(G)$ is
	\begin{equation}
\diag[\underbrace{1,\ldots,1}_{7},\underbrace{2,\ldots,2}_{4},8,967498002648]
\end{equation}
and $\Delta(G)=2049840225216075785191098057600067625877504$. Although both $\Delta(G)$ and $d_n$ are huge, the great common divisor of these two numbers is small. Indeed, it can be obtained that
$\gcd(d_n,\Delta(G))=72=2^3 \times 3^2$. Using the basic estimate of $\ell(G)$, we obtain
$\ord_2(\ell(G))\le \ord_2(d)=3$, $\ord_3(\ell(G))\le \ord_3(d)=2$, and $\ord_p(\ell(G))=0$ for any prime other than 2 and 3. Thus, $\ell(G)$ divides 72.

\subsection{Improved estimates}
	Let $p$ be an odd prime and  $f\in \mathbb{F}_p[x]$ be a monic polynomial over the field $\mathbb{F}_p$. Now let $f = \prod_{1\le i\le r}f_i^{e_i}$
be the irreducible factorization of $f$, with distinct monic irreducible polynomials $f_1,f_2,\ldots, f_r$ and positive integers $e_1,e_2,\ldots, e_r$. The \emph{square-free part} of $f$, denoted by $\sfp (f)$,  is $\prod_{1\le i\le r}f_i$, see \cite[p.~394]{gathen}.
\begin{definition}[\cite{wang2021EUJC}]\normalfont{
	Let $G$ be a graph and $p$ be an odd prime. The $p$-\emph{main} polynomial of $G$, denoted by $m_p(G;x)$,  is the monic polynomial $f\in \mathbb{F}_p[x]$ of smallest degree such that $f(A)e=0$.}
\end{definition}

In \cite{wangzhu}, it was noted that two generalized cospectral graphs may have different $p$-main polynomials for some prime $p$. However, the authors proved that under certain conditions, the $p$-main polynomial of $G$ is constant among all  generalized cospectral mates, which reads
\begin{equation}\label{mpequ}
m_p(G;x)=\frac{\chi(G;x)}{\sfp(\gcd(\chi(A;x),\chi(A+J;x)))},
\end{equation}
where  $\chi(G;x)=\chi(A;x)=\det (xI-A)$ is the characteristic polynomial of $G$ and all computations are done over $\mathbb{F}_p$.  For simplicity, we use $\tilde{m}_p(G;x)$ to denote the right hand side in Eq.~\eqref{mpequ}. Of course,  $\tilde{m}_p(G;x)$ is invariant under generalized cospectrality. The following two lemmas were essentially proved in \cite{wangzhu}.
\begin{lemma}
	$n\ge \deg \tilde{m}_p(G;x)\ge \deg m_p(G;x)=\rank_p\,W(G)$, where the first equality is strict if and only if $\rank_p\,W(G)<n$.
\end{lemma}
\begin{lemma}\label{ubm}
	Let $G$ and $H$ be generalized cospectral mates with adjacent matrices $A$ and $B$, respectively. Write $\tilde{m}_p(x)=\tilde{m}_p(G;x)=\tilde{m}_p(H;x)$. Then $\tilde{m}_p(A)e=\tilde{m}_p(B)e=0$ over $\mathbb{F}_p$.
\end{lemma}
We mention that Lemma \ref{ubm} even holds for $p=2$, but the special case $p=2$ had been settled more effectively in \cite{qiuarXiv}.
\begin{lemma}[\cite{qiuarXiv}]\label{ubm2}
 Let $\chi(G;x) = x^n + c_1x^{n-1} + \cdots + c_{n-1}x + c_n$ be the characteristic polynomial of graph $G$. Define $k=\lfloor n/2\rfloor$ and
 \begin{equation}\psi(x)= \psi(G;x)=\begin{cases}
  x^k + c_2x^{k-1} +c_4x^{k-2}+\cdots+c_n&\text{if~} n \text{~is even},\\
 x(x^k + c_2x^{k-1} +c_4x^{k-2}+ \cdots+ c_{n-1})&\text{if~} n \text{~is odd.}
 \end{cases}
 \end{equation}
 Then $\psi(A)e\equiv 0\pmod 2$.
\end{lemma}
We emphasize that $\psi(G;x)$ is invariant under cospectrality, and hence is invariant under generalized cospectrality of course. For convenience, we define
\begin{equation}
M_p(G;x)\equiv\begin{cases}
\tilde{m}_p(G;x)\pmod{p}&\text{if~} p \text{~is an odd prime},\\
\psi(G;x)\pmod{2}&\text{if~} p=2.
\end{cases}
\end{equation}
To be specific, we assume all the coefficients of $M_p(G;x)$ belong to $\{0,1,\ldots,p-1\}$.
Now we introduce a  modified walk matrix for a graph $G$ associated with a prime $p$. We first consider the case that $G$ is controllable.
\begin{definition}\normalfont{
	Let $G$ be a controllable graph of order $n$ and $p$ be a prime with $\rank_p W(G)\le n-1$. Then the modified walk matrix of $G$ with respect to $p$, denoted by $W^{(p)}(G)$, is
	\begin{equation}
	\left[e,Ae,\ldots,A^{s-1}e, \frac{M_p(A)e}{p},\frac{AM_p(A)e}{p},\ldots,\frac{A^{n-1-s}M_p(A)e}{p}\right],
	\end{equation} }
where $A=A(G)$, $M_p(x)=M_p(G;x)$ and $s$ is the degree of $M_p(x)$.
\end{definition}
Note that $W^{(p)}(G)$ is integral by Lemmas \ref{ubm} and \ref{ubm2}. We shall show that the last invariant factor of $W^{(p)}(G)$ is no more than that
of $W(G)$.
\begin{lemma}\label{dndec}
	Let $M=[\alpha_1,\alpha_2,\ldots,\alpha_n]$ be a nonsingular integral $n\times n$ matrix. Suppose that $\alpha_k\equiv 0\pmod{p}$ for some prime $p$ and positive integer $k\in\{1,\ldots,n\}$. Let $\tilde{M}=[\alpha_1,\ldots,\alpha_{k-1},\frac{1}{p}\alpha_k,\alpha_{k+1},\ldots,\alpha_n]$. Then $\ord_p(d_n(\tilde{M}))\le \ord_p(d_n({M}))$.
\end{lemma}
\begin{proof}
	 Let $\Delta_i$   and  $\tilde{\Delta}_i$ denote the $i$-th determinantal divisor of $M$ and $\tilde{M}$ respectively.  Let $D_{n-1}$ be any $(n-1)$-th minor of $M$ and $\tilde{D}_{n-1}$ be the corresponding minor of $\tilde{M}$. It is easy to see that  $\tilde{D}_{n-1}$ equals $D_{n-1}$ or $\frac{1}{p}D_{n-1}$. This means $\tilde{\Delta}_{n-1}$ equals  $\Delta_{n-1}$ or $\frac{1}{p}\Delta_{n-1}$. On the other hand, we have   $\tilde{\Delta}_n=|\det \tilde{M}|=\frac{1}{p}|\det M|=\frac{1}{p}\Delta_n$. It follows that
	  $d_n(\tilde{M})$ equals $d_n(M)$ or $\frac{1}{p}d_n(M)$. Thus, $\ord_p(d_n(\tilde{M}))\le \ord_p(d_n({M}))$ as desired.
\end{proof}
\begin{lemma}
	Let $G$ be controllable and $p$ be a prime with $\rank_p(W(G))\le n-1$. Then we have $$\ord_p(d_n(W^{(p)}(G)))\le \ord_p(d_n(W(G))).$$
\end{lemma}
\begin{proof}
	Let $W=W(G)$ and $\hat{W}=[e,Ae,\ldots,A^{s-1}e,M_p(A)e,AM_p(A)e,\ldots,A^{n-1-s}M_p(A)e]$. We note that $\hat{W}$ and $W$ have the same Smith normal form since $\tilde{W}=WU$ for some  unit upper triangular  (and hence unimodular) matrix $U$. This means $d_n(\hat{W})=d_n(W)$. Using Lemma \ref{dndec} iteratively, we obtain $\ord_p(d_n(W^{(p)}(G)))\le \ord_p(d_n(\hat{W}))$,  which completes the proof.
\end{proof}
For controllable graphs, Lemma \ref{lqc} can be improved as follows using the modified walk matrix. An initial version of this result first appeared in \cite{qiuarXiv}.
\begin{lemma}\label{ordimp}
	If $G$ be a controllable graph, then $\ord_p(\ell(G))\le \ord_p(d_n(W^{(p)}(G)))$ for any prime $p\mid \det W(G)$.
\end{lemma}
\begin{proof}
	Let $Q\in \mathcal{Q}(G)$ and $H$ be the graph with adjacency matrix $A(H)=Q^\T A(G)Q$. Then $Q^\T W(G)=W(H)$. As $M_p(G,x)$ is invariant under generalized cospectrality, we easily find that the two modified walk matrices $W^{(p)}(G)$ and $W^{(p)}(H)$ are obtained from their original walk matrices by the same collum operations. Thus, from the equality $Q^\T W(G)=W(H)$ we obtain $Q^\T W^{(p)}(G)=W^{(p)}(H)$. Consequently, $\ell(Q)\mid d_n(W^{(p)}(G))$ by the same argument of Lemma \ref{lqc}.
\end{proof}

\begin{remark}\normalfont{
	In \cite{wang2017JCTB}, the author proved that if $2^{-\lfloor\frac{n}{2}\rfloor}\det W(G)$ is odd and square-free then $G$ is DGS. It is not hard to show that under the same condition $W^{(p)}(G)$ always has full rank for any prime factor $p$ of $\det W$. Consequently, $\ord_p(\ell(G))\le \ord_p(d_n^{(p)})=0$ for any prime $p$, that is, $\ell(G)=1$. Thus,  the result in \cite{wang2017JCTB} can be recovered by Lemma \ref{ordimp}. }
\end{remark}
		Consider the same graph as in Example 1. Direct calculation (using Mathematica) shows that the Smith normal forms of $W^{(2)}(G) $ and $W^{3}(G)$ are
		\begin{equation}
		\diag[\underbrace{1,\ldots,1}_{11},4,483749001324]\quad\text{and}\quad\diag[\underbrace{1,\ldots,1}_{7},\underbrace{2,\ldots,2}_{4},8,322499334216],
		\end{equation}
		respectively. By Lemma  \ref{ordimp}, we obtain
		\begin{equation}
		\ord_2(\ell(G))\le \ord_2(483749001324)=2 \quad \text{and}\quad	\ord_3(\ell(G))\le \ord_3(322499334216)=1.
		\end{equation} Thus, $\ell(G)$ divides 12, improving the previous result that $\ell(G)$ divides 72.

\begin{definition}\label{mwa}\normalfont{
		Let $G$ be an almost controllable graph of order $n$ and $p$ be a prime with $\rank_p W_0(G)\le n-1$. Then the modified walk matrix of $G$ with respect to $p$, denoted by $W_\delta^{(p)}(G)$, is
		\begin{equation}\label{wdp}
		\left[e,Ae,\ldots,A^{s-1}e, \frac{M_p(A)e}{p},\frac{AM_p(A)e}{p},\ldots,\frac{A^{n-2-s}M_p(A)e}{p},(-1)^\delta\frac{1}{\lcm(2^{\lfloor\frac{n}{2}\rfloor-1},p^{n-1-s})}\xi(G)\right] ,
		\end{equation} }
	where $A=A(G)$, $M_p(x)=M_p(G;x)$ and $s$ is the degree of $M_p(x)$.
\end{definition}

In Definition \ref{mwa}, we make the convention that  $W_\delta^{(p)}(G)=W_\delta(G)$ if the degree of $M_p(x)$ is $n-1$. We remark that $\deg M_2(x)=\lceil\frac{n}{2}\rceil$ and hence
\begin{equation}
\lcm(2^{\lfloor\frac{n}{2}\rfloor-1},p^{n-1-s})=
\begin{cases}
 2^{\lfloor\frac{n}{2}\rfloor-1}& \text{if~} p=2,\\
 	2^{\lfloor\frac{n}{2}\rfloor-1}p^{n-1-s}& \text{otherwise}.
\end{cases}
\end{equation}
Recall that the $i$-th entry of $\xi(G)$ is the algebraic cofactor of the $(i,n)$-entry for $W(G)$. Since $M_p(x)$ is a monic polynomial of degree $s$, every such algebraic cofactor for $W(G)$ equals the corresponding algebraic cofactor for
$$W'=[e,Ae,\ldots,A^{s-1}e,M_p(A)e,AM_p(A)e,\ldots,A^{n-2-s}M_p(A)e,A^{n-1}e].$$
Since all entries in the $n-1-s$ columns $ M_p(A)e$, $AM_p(A)e$, $\ldots$, $A^{n-2-s}M_p(A)e$ are multiple of $p$, we find that all algebraic cofactors for entries in the last column of $W'$ are multiple of $p^{n-1-s}$. This verifies that each entry of the last column in Eq. \eqref{wdp} is indeed an integer.

Lemma \ref{lqa} can be improved in the same way.
\begin{lemma}\label{ordimpa}
	If $G$ is almost controllable then $\ord_p(\ell(G))\le \ord_p(d_n(W_0^{(p)}(G)))$ for any prime $p\mid \det W_0(G)$.
\end{lemma}
Using Lemmas \ref{ordimp} and \ref{ordimpa}, we obtain the following improvement of Proposition \ref{lG1}, which we shall refer to as improved estimate of $\ell(G)$.
\begin{proposition}\label{lG2}
	Let $G$ be controllable or almost controllable. Write
		\begin{equation}
	d=\begin{cases}
	d_n(W(G))&\text{if~} G\text{~is controllable,}\\
	d_n(W_0(G))&\text{if~} G\text{~is almost controllable.}
	\end{cases}
	\end{equation}
	and
	\begin{equation}
	d^{(p)}=\begin{cases}
	d_n(W^{(p)}(G))&\text{if~} G\text{~is controllable,}\\
	d_n(W_0^{(p)}(G))&\text{if~} G\text{~is almost controllable.}
	\end{cases}
	\end{equation}
	Then we have\\	
	\textup{(\rmnum{1})} $\ord_2(\ell(G))\le \ord_2(d^{(2)})$;\\	
	\textup{(\rmnum{2})} for odd prime $p$, $\ord_p(\ell(G))\le \ord_p(d^{(p)})$ if $p\mid\gcd(\Delta(G),d)$ and $p^2\mid\Delta(G)$; otherwise $\ord_p(\ell(G))=0$.
\end{proposition}
\section{Finding all generalized cospectral mates}\label{criterion}
We first deal with controllable graphs. Given a controllable graph $G$, a prime $p$ is \emph{eliminable} if we can guarantee $\ord_p(\ell(G))=0$ using the improved estimate of $\ell(G)$, and is \emph{potential} otherwise.  Throughout this section, we use the following notations:
\begin{equation}
t_i=\ord_{p_i}(d_n(W^{(p_i)}(G))) ~for~ i\in\{1,2,\ldots,s\}, ~and~L=\prod_{i=1}^s p_i^{t_i},
\end{equation}
where $p_1,p_2,\ldots,p_s$ are all potential primes associated with $G$. If there are no potential primes for $G$, then $s=0$ and $L=1$.
\subsection{Generalized cospectral mates and maximum cliques}
\begin{lemma}\label{nontrsol}
Let $G$ be a controllable graph. If $G$ has a generalized cospectral mate, then $s>0$ and the following system of equations
\begin{equation}\label{sys1}
\left\{
\begin{aligned}
&(W^{(p_i)}(G))^\T x\equiv 0\pmod {p^{t_i}_i} \text{~for~} i=1,2,\ldots,s,\\
&x^\T x=L^2,&\\
&e^\T x=L,&\\
&x^\T A(G)x=0&
\end{aligned}
\right.
\end{equation}
has an integral solution $x\not\equiv 0\pmod{L}$.
\end{lemma}
\begin{remark}\normalfont{
	Let $e^{(k)}$ denote the $k$-th unit coordinate vector for $k=1,\ldots,n$. Then  $Le^{(k)}$'s are always solutions of Sys. \eqref{sys1}. These solutions are called \emph{trivial}, whereas any other solutions, that is, $x\not\equiv 0\pmod{L}$, are called  \emph{nontrivial}. }
\end{remark}
\begin{remark}
	Sys. \eqref{sys1} becomes trivial when $s=0$, i.e., $L=1$. In this case, the solution set of Sys. \eqref{sys1} is exactly $\{e^{(1)},\ldots,e^{(n)}\}$.
\end{remark}
\begin{proof}
	As $G$ has a generalized cospectral mate, $\mathcal{Q}(G)$ contains a matrix that is not a permutation matrix. Let $Q$ be such a matrix and $\ell$ be its level. Then, $\ell\neq 1$. Let $\xi$ be a collum of $Q$ which contains at least a non-integer entry. Write $x=L\xi$.  By Proposition \ref{lG2}, we see that $\ell\mid L$, which implies that $L>1$ and hence $s\ge 1$. Furthermore, as $\ell \xi$ is integral and $\xi$ contains some non-integer entry, we find that $x$ is  an integral vector and $x\not\equiv 0\pmod{L}$.
	
	Since $Q^\T Q=I$ and $e^\T Q=e$, we see that $x^\T x=L^2$ and $e^\T x=L$. Let $H$ be the generalized cospectral mate of $G$ corresponding to $Q$, i.e., $A(H)=Q^\T A(G)Q$. As each diagonal entry of $H$ is zero, we have $\xi^\T A(G)\xi=0$ and hence $x^\T A(G)x=0$. As $(W^{(p_i)}(G))^\T Q=(W^{(p_i)}(H))^\T$, we easily find $(W^{(p_i)}(G))^\T (LQ)\equiv 0\pmod{L}$ and in particular, $(W^{(p_i)}(G))^\T x\equiv 0\pmod{L}$. Thus $(W^{(p_i)}(G))^\T x\equiv 0\pmod{p_i^{t_i}}$ as $p_i^{t_i}\mid L$.
\end{proof}
\begin{theorem}\label{ncc}
	Let $G$ be a controllable graph. Then $G$ has a generalized cospectral mate if and only if Sys. \eqref{sys1} has $n$ solutions $\xi_1,\ldots,\xi_n$, not all trivial, such that $\xi_i^\T \xi_j=0$ and $\xi_i^\T A(G)\xi_j\in \{0,L^2\}$ for all distinct $i$ and $j$.
\end{theorem}
\begin{proof}
	The `only if' part is clear using the same argument of Lemma \ref{nontrsol}. To show the `if' part, we set $Q=\frac{1}{L}[\xi_1,\ldots,\xi_n]$. It is routine to check that $Q$ is a rational regular orthogonal matrix and $Q^\T A(G)Q$ is a 0-1 matrix with vanishing diagonals. Since at least one of $\xi_i$'s are nontrivial, $Q$ is not the permutation matrix. This means the graph with adjacency matrix $Q^\T A(G)Q$ is not isomorphic to $G$ and hence is a generalized cospectral mate of $G$.
\end{proof}
The following definition suggested by Theorem \ref{ncc} is crucial.
\begin{definition}\label{OG}\normalfont{
	Let $G$ be a controllable graph.  Define a graph $\Omega(G)$ whose vertex set is the solution set of \eqref{sys1}; two vertices $\xi$ and $\eta$ are adjacent if $\xi^\T\eta=0$ and $\xi^\T A(G) \eta\in\{0,L^2\}$.}
\end{definition}
We remark that if Sys. \eqref{sys1} has no nontrivial solutions, then $\Omega(G)$ is $K_n$, the complete graph of order $n$. We note that $\Omega(G)$ always contains a clique of order $n$ consisting of the $n$ trivial solutions $Le^{(k)}$, $k=1,\ldots,n$. In particular, $\Omega(G)$ is necessarily $K_n$ when $L=1$.  We also note that there exist no cliques of order $n+1$ in $\Omega(G)$ since vertices in any clique are necessarily pairwise orthogonal as vectors in Euclidean space $\mathbb{R}^n$.

We say two matrices $Q_1$ and $Q_2$  are column permutation equivalent, if there exists a permutation matrix $P$ such that $Q_2=Q_1P$. For a controllable graph $G$, every generalized cospectral mate of $G$ (together with $G$) corresponds to a unique $Q$ in $ \mathcal{Q}(G)$. Two generalized cospectral mates of $G$ are isomorphic if and only if the corresponding $Q$'s are column permutation equivalent, that is, the corresponding maximum cliques are the same. This leads to the following result, which is a refinement of Theorem  \ref{ncc}.
\begin{theorem}\label{nsc2}
	Let $G$ be a controllable graph. Let $\mathcal{C}(G)$ be the set of all unlabeled graphs generalized cospectral with $G$. Then there is a one-to-one correspondence between $\mathcal{C}(G)$ and the maximum cliques of $\Omega(G)$.
\end{theorem}
The significance of Theorem \ref{nsc2} is that it gives a method to generate $\mathcal{C}(G)$ and hence all its generalized cospectral mates. In particular, it gives the following complete criterion to determine whether a controllable graph is DGS or not.
\begin{corollary}\label{dgscon}
	Let $G$ be a controllable graph. Then $G$ is DGS if and only if $\Omega(G)$ has a unique maximum clique.
\end{corollary}
Note that if $\Omega(G)$ is the complete graph $K_n$, then $\Omega(G)$ trivially has a unique maximum clique. Thus the following corollary is immediate.
\begin{corollary}
Let $G$ be a controllable graph of order $n$. If  $\Omega(G)$ is $K_n$ then $G$ is DGS.
\end{corollary}

\subsection{Solving the equations}\label{threesteps}
We shall solve Sys. \eqref{sys1} in three steps:

\noindent\emph{Step 1:} For each $p_i$ $(i=1,2,\ldots,s)$, solve the following system of congruence equations
\begin{equation}\label{sysi}
\left\{
\begin{aligned}
&(W^{(p_i)}(G))^\T \xi_i\equiv 0\pmod {p_i^{t_i}},& \\
&\xi_i^\T \xi_i\equiv 0\pmod{p_i^{t_i}},&\\
&\xi_i^\T A(G)\xi_i\equiv 0\pmod {p_i^{t_i}}.&
\end{aligned}
\right.
\end{equation}

\noindent\emph{Step 2:} Let $\Pi_i$ denote the solution set of Sys. \eqref{sysi} for each $i$. For each element $(\xi_1,\xi_2,\ldots,\xi_s)$ in the Cartesian product $\Pi_1\times \Pi_2\times\cdots\times\Pi_s$, solve the system of congruence equations for vector $\eta$:
 \begin{equation}\label{syscrt}
\left\{
\begin{aligned}
&\eta \equiv \xi_1\pmod {p_1^{t_1}},& \\
&\eta\equiv \xi_2\pmod{p_2^{t_2}},&\\
&\cdots\\
&\eta \equiv \xi_s\pmod{p_s^{t_s}}.&
\end{aligned}
\right.
\end{equation}
\noindent\emph{Step 3:} For each solution $\eta$ obtained in Step 2, find all solutions of the following  equation
\begin{equation}\label{sysdio}
\left\{
\begin{aligned}
&x \equiv \eta\pmod {L}, \\
&x^\T x=L^2, \\
&e^\T x=L,\\
&x^\T A(G) x=0.
\end{aligned}
\right.
\end{equation}
We recall two standard tools for solving the equations in the first two steps. The following lemma is a standard application of Smith normal form. See, e.g., \cite{newman1972}
\begin{lemma}\label{allsol}
	Let $M$ be an $n\times n$ integral matrix and $m$ be a positive integer. Let $U$ and $V$ be unimodular integral matrices such that $UMV=\diag [d_1,d_2,\ldots,d_n]$, the Smith normal form of $M$. Then solutions of
	$Mx\equiv 0\pmod m$  can be expressed as
	$\sum_{i=1}^{n}k_i\frac{m}{\gcd(m,d_i)}Ve^{(i)}$, where $k_i\in \{0,1,\ldots,\gcd(m,d_i)-1\}$ and $e^{(i)}$ is the $i$-th  unit coordinate vector.
\end{lemma}
\begin{proof}
	Let $x=Vy=V(y_1,\ldots,y_n)^\T$. Then the equation $Mx\equiv 0\pmod {m}$ is equivalent to \begin{equation}\label{equ_y}
	(\diag[d_1,d_2,\ldots,d_n])y\equiv 0\pmod{m}.
	\end{equation}
	That is, $d_iy_i\equiv 0\pmod m$ for $i\in \{1,2,\ldots,n\}$. Note that  $d_iy_i\equiv 0\pmod m$ has exactly $\gcd(m,d_i)$ different solutions (modulo $m$), which can be expresses as $k_i\frac{m}{\gcd(m,d_i)}$, for $k_i\in\{0,1,\ldots,\gcd(m,d_i)-1\}$. Thus, the solution set of Eq. $\eqref{equ_y}$ is 	$\sum_{i=1}^{n}k_i\frac{m}{\gcd(m,d_i)}e^{(i)}$. Noting that $x=Vy$, the proof is complete.
\end{proof}
\begin{corollary}\label{numsl}
	$(W^{(p_i)}(G))^\T \xi_i\equiv 0\pmod {p_i^{t_i}}$ has exactly
	$$p_i^{\ord_{p_i}(\det W^{(p_i)}(G))}$$
		different solutions (modulo $p_i^{t_i}$).	
\end{corollary}
\begin{proof}
	Let $d_k$ be the $k$-th invariant factor of $W^{(p_i)}(G)$. Note that $t_i=\ord_{p_i}(d_n)$. Using Lemma \ref{allsol}, the number of solutions to the congruence 	$(W^{(p_i)}(G))^\T \xi_i\equiv 0\pmod {p_i^{t_i}}$ is
	$$\prod_{k=1}^n \gcd(p_i^{t_i},d_k)=\prod_{k=1}^n p_i^{\ord_{p_i}(d_k)}= p_i^{\sum_{k=1}^n\ord_{p_i}(d_k)}=p_i^{\ord_{p_i}(\prod_{k=1}^nd_k)}=p_i^{\ord_{p_i}(\det W^{(p_i)}(G))}.$$
	This completes the proof.
\end{proof}
\begin{lemma}[Chinese Remainder Theorem]
	Let $m_1,m_2,\ldots,m_s$ be pairwise relatively prime positive integers. Then the system of congruences
	\begin{equation}\label{syscrtt}
	\left\{
	\begin{aligned}
	&x\equiv a_1\pmod {m_1},& \\
	&x\equiv a_2\pmod {m_2},&\\
	&\cdots	&&\\
	&x\equiv a_s\pmod {m_s}&\\
	\end{aligned}
	\right.
	\end{equation}
	has a unique solution modulo $M= m_1m_2 \cdots m_s$. Moreover, the unique solution can be expressed as
	$$x\equiv a_1M_1y_1+a_2M_2y_2+\cdots+a_sM_sy_s\pmod{M},$$
	where $M_i=\frac{M}{m_i}$ and $y_i$ is the inverse of $M_i$ modulo $m_i$.
\end{lemma}
It remains to deal with the system of equations in Step 3.  The basic technique comes from \cite{wangyu}, where the authors considered the special case that $m$ is an odd prime.
\begin{definition}\normalfont{
	Let $v$ be an integral vector. For a positive integer $m$, an integral vector $w$ is  called a \emph{perfect $m$-representative} of $v$ if $w\equiv v\pmod{m}$, $e^\T w=m$ and $w^\T w=m^2$. We say an integral vector $v$ is \emph{$m$-perfect} if it has at least one perfect $m$-representative.}
\end{definition}
\begin{definition}\label{sv}\normalfont{
		Let  $v$ be an integral vector. For a positive integer $m$, an integral vector $w=(w_1,\ldots,w_n)^{\T}$ is  called  \emph{the shortest  $m$-representative} of $v$ if $w\equiv v\pmod{m}$ and $-\frac{m}{2}<w_i\le \frac{m}{2}$ for $1\le i\le n$.}
\end{definition}
We note that the shortest $m$-representative of $v$ is always unique whereas there may be none, one, or many perfect $m$-representatives of $v$. For example, if $v$ is an $n$-dimensional vector with each entry zero modulo $m$,  then the shortest $m$-representative of $v$ is the zero vector; whereas $v$ has exactly $n$ perfect $m$-representatives: $(m,0,0\ldots,0)$, $(0,m,0,\ldots,0)$, $\ldots$, $(0,\ldots,0,m)$.

 For two vectors $u$ and $v$, the  \emph{Hamming distance} of $u$ and $v$ is the number of positions in which they differ. The following lemma indicates that for any integral vector $v$, all perfect $m$-representatives of $v$ are very close to its shortest $m$-presentative in the sense of Hamming distance. This fact was established in \cite{wangyu} for the case that $m$ is an odd prime.
\begin{lemma}\label{near}
Let $v=(v_1,\ldots,v_q;v_{q+1},\ldots,v_n)$ be an integral vector where $0< q\le n$, $v_i\not\equiv 0\pmod{m}$ for $i\in\{1,\ldots,q\}$ and $v_j\equiv 0\pmod{m}$ for $j\in\{q+1,\ldots,n\}$. Let $w=(w_1,w_2,\ldots,w_n)$ be a perfect $m$-presentative of $v$ and  $u=(u_1,\ldots,u_n)$  be the shortest $m$-presentative. Then the followings hold:

\noindent\textup{(\rmnum{1})} $u_i,w_i\neq 0$ for $i\in\{1,2,\ldots,q\}$, and $u_j=w_j=0$ for $j\in\{q+1,\ldots,n\}$;

\noindent\textup{(\rmnum{2})} The Hamming distance of $(w_1,w_2,\ldots,w_q)$ and $(u_1,u_2,\ldots,u_q)$ is at most $3$. Moreover, for any $i\in\{1,2,\ldots,q\}$ such that $w_i\neq u_i$, either $w_i=u_i-m$ and $u_i>0$, or $w_i=u_i+m$ and $u_i<0$.
\end{lemma}
\begin{proof}
Let $i\in \{1,2,\ldots,q\}$. As $w_i\equiv u_i\equiv v_i$ and $v_i\not\equiv 0\pmod{m}$, we find that $u_i,v_i\neq 0$. Let $j\in\{q+1,\ldots,n\}$.	As $v_j\equiv 0\pmod{m}$, we see that $u_j=0$ by Definition \ref{sv}. Note that $w_j\equiv v_j\equiv 0\pmod{m}$. Suppose to the contrary that $w_j\neq 0$. Then $w_j=km$ for some nonzero integer $k$, implying $w_j^2\ge m^2$. It follows that $w^\T w\ge w_1^2+w_j^2>m^2$, a contradiction. This proves (\rmnum{1}).

Let $I=\{i\colon\,1\le i\le q\text{~and~}w_i\neq u_i\}$. Let $i\in I$.  As $-m/2<v_i\le m/2$ and $w_i=u_i+km$ for some nonzero integer $k$, we find that $|w_i|\ge m/2$ with equality holding if and only if $v_i=m/2$ and $w_i=-m/2$. Suppose to the contrary that $|I|\ge 4$. Then we have $$w^\T w\ge \sum_{i\in I}|w_i^2|\ge 4\times \left(\frac{m}{2}\right)^2\ge m^2.$$
This means all equality must hold simultaneously. Thus $|I|=q=4$ and $w_i=-m/2$ for each $i\in I$. Consequently, we have $e^\T w=-2m$, a contradiction. This proves $|I|\le3$. The remaining part of (\rmnum{2}) clearly follows from the trivial inequality that $|w_i|\le m$.
\end{proof}
\begin{corollary}\label{necforperfect}
	Let $v\not\equiv 0\pmod{m}$ and $w$ be the shortest $m$-representative of $v$. If $v$ is $m$-perfect then $w^\T w\le m^2$ and  $|e^\T w-m|\le 3m$.
\end{corollary}
Lemma \ref{near} gives a simple way to solve system \eqref{sysdio} in Step 3. Indeed, the first three equations in \eqref{sysdio} mean that $x$ is an $L$-perfect representative of $\eta$. Since the case that $\eta\equiv 0\pmod{L}$ is trivial, we assume $\eta\not\equiv 0\pmod{L}$. Let $\hat{\eta}$ be the shortest $L$-representative of $\eta$.  If $\hat{\eta}^\T \hat{\eta}>L^2$ or $|e^\T \hat{\eta}-L|>3L$ then $\eta$ is not $L$-perfect by Corollary \ref{necforperfect} and hence  Sys.~\eqref{sysdio} has no solutions for this $\eta$. Otherwise, we can search all possible $L$-perfect representatives of $\eta$  in $O(q^3)$ times by Lemma \ref{near}, where $q$ is the number of entries in $\xi$ which are nonzero modulo $L$. By checking the last equation for each  $L$-perfect representative $x$, we obtain the solution set of \eqref{sysdio}.

In \cite{wangyu}, the authors gave a more refined analysis on the Hamming distance between perfect representatives and the shortest representative, which can be used to speed up the searching for perfect representatives. We refer the reader to \cite{wangyu} for details.
\subsection{An example}
As an illustration, we try to find all generalized cospectral mates for the graph in Example \ref{g13}. We already know that $L=2^2\times 3^1=12$. For this graph,  Sys. \eqref{sys1} becomes
\begin{equation}\label{sys1forexam}
\left\{
\begin{aligned}
&(W^{(2)}(G))^\T x\equiv 0\pmod {4},& \\
&(W^{(3)}(G))^\T x\equiv 0\pmod {3},&\\
&x^\T x=144,&\\
&e^\T x=12,&\\
&x^\T A(G)x=0.&
\end{aligned}
\right.
\end{equation}
In Step 1, we  solve the following two congruence systems (of variables $\xi_1$ and $\xi_2$ respectively):
\begin{equation}\label{sysiforexam}
\left\{
\begin{aligned}
&(W^{(2)}(G))^\T \xi_1\equiv 0\pmod {4},& \\
&\xi_1^\T \xi_1\equiv 0\mod{4},&\\
&\xi_1^\T A(G)\xi_1\equiv 0\pmod {4};&
\end{aligned}
\right.\quad \text{and}\quad \left\{
\begin{aligned}
&(W^{(3)}(G))^\T \xi_2\equiv 0\pmod {3},& \\
&\xi_2^\T \xi_2\equiv 0\mod{3},&\\
&\xi_2^\T A(G)\xi_2\equiv 0\pmod {3}.&
\end{aligned}
\right.
\end{equation}
Consider the equation $(W^{(2)}(G))^\T \xi_1\equiv 0\pmod {4}$. Using Lemma \ref{allsol}, all solutions for this equation can be expressed as

$$\xi_1=c_1\alpha+c_2\beta,$$
where $c_1,c_2\in\{0,1,2,3\}$, $\alpha=(1, 2, 2, 0, 1, 1, 2, 2, 3, 2, 3, 1, 0)^\T$  and $\beta=(0, 1, 3, 2, 3, 3, 2, 1, 2, 2,$
$3, 1, 1)^\T$. Among these 16 solutions, 8 of them satisfy the remaining two equations for the left of \eqref{sysiforexam}. We record the solutions in the first column of Table \ref{CRTforexam}. Similarly, we find that the right hand side of \eqref{sysiforexam} has exactly 3 solutions, as described in the first row of Table \ref{CRTforexam}.
	
In Step 2, we solve the system of congruences \begin{equation}\label{syscrt2}
\left\{
\begin{aligned}
&\eta \equiv \xi_1\pmod {4}, &\\
&\eta \equiv \xi_2\pmod{3}&
\end{aligned}
\right.
\end{equation}
using the Chinese Remainder Theorem for every entry of $\eta$. We have $\eta\equiv 9\xi_1+4\xi_2\pmod{12}$. The solutions to \eqref{syscrt2} for all combinations of $\xi_1$ and $\xi_2$ are listed in Table \ref{CRTforexam}.
\begin{table} 	
	\centering \caption{Illustration of Step 2.}\label{CRTforexam}
	\begin{threeparttable}
		\begin{tabular}{c|ccccc} \toprule \diagbox[width=12em] {$\quad\quad\quad\quad\xi_1$}{$\xi_2\quad\quad\quad$} &{0000000000000}&&{0010021000122}&&{0020012000211}\\\hline
			{0000000000000}&{0000000000000}&&{0040084000488}&&{0080048000844}\\
			{0132332122311}&{0936336966399}&&{09763ba966755}&&{09b6372966b11}\\
			{0220220200222}&{0660660600666}&&{06a0624600a22}&&{06206a86002aa}\\
			{0312112322133}&{0396996366933}&&{031695a3661bb}&&{0356912366577}\\
			{2000220020220}&{6000660060660}&&{6040624060a28}&&{60806a80602a4}\\
			{2132112102131}&{6936996906939}&&{697695a9061b5}&&{69b6912906571}\\
			{2220000220002}&{6660000660006}&&{66a0084660482}&&{662004866084a}\\
			{2312332302313}&{6396336306393}&&{63163ba30675b}&&{6356372306b17}\\
			\bottomrule \end{tabular}
		\begin{tablenotes}
			\footnotesize
			\item[*]$a=10,b=11.$
		\end{tablenotes}
	\end{threeparttable} \label{tab:Training_sizes} \end{table}

In Step 3, for each $\eta$, we generate all possible perfect $12$-representatives based on Lemma \ref{near}. We leave all such representatives unchanged if they also satisfy the equation $x^\T Ax=0$. This gives the solution set of \eqref{sys1forexam}.  We record the solutions as columns in the following matrix.

\begin{equation}\small{
X=
	\left(
	\begin{array}{ccccccccccccc}
  &    &  &  0& 0 & 0 & 0 & 0 & 0 & 0 & 0 & 0 & 0 \\
 &    &  & 0 & 0 & 0 & 0 & 0 & 0 & -6 & 6 & 6 & 6  \\
 &    &  &4 & 4 & 4 & 8 & -4 & -4 & -2 & -2 & -2 & 2 \\
  &    &  & 0 & 0 & 0 & 0 & 0 & 0 & 0 & 0 & 0 & 0 \\
 &    &  & 0 & 0 & 0 & 0 & 0 & 0 & 6 & -6 & 6 & 6 \\
 &    &  & 8 & -4 & -4 & 4 & 4 & 4 & 2 & 2 & 2 & -2 \\
12e^{(1)},  & \cdots, &12 e^{(13)}, & 4 & 4 & 4 & -4 & 8 & -4 & 4 & 4 & 4 & -4 \\
 &    &  & 0 & 0 & 0 & 0 & 0 & 0 & 6 & 6 & -6 & 6 \\
 &   &  & 0 & 0 & 0 & 0 & 0 & 0 & 0 & 0 & 0 & 0 \\
 &    &  & 0 & 0 & 0 & 0 & 0 & 0 & 0 & 0 & 0 & 0 \\
 &    &  &  4 & 4 & 4 & -4 & -4 & 8 & -2 & -2 & -2 & 2\\
 &    &  & -4 & 8 & -4 & 4 & 4 & 4 & 2 & 2 & 2 & -2 \\
  &    &  & -4 & -4 & 8 & 4 & 4 & 4 & 2 & 2 & 2 & -2 \\
	\end{array}
	\right)_{13\times 23}.}
\end{equation}
We write the columns of $X$ as $x^{(1)},\ldots,x^{(23)}$. Now the graph $\Omega(G)$ has vertices $x^{(1)},\ldots, x^{(23)}$; two vertices $x^{(i)}$ and $x^{(j)}$ are adjacent whenever $(x^{(i)})^\T x^{(j)}=0$ and $(x^{(i)})^\T A x^{(j)}=0$ or 144. Using Mathematica, we find that the graph $\Omega(G)$ has exactly $3$ cliques of order 13 (including the trivial clique consisting of the $e^{(i)}$'s). The two nontrivial cliques of order 13 are
$$C_1=\{x^{(1)}, x^{(4)},x^{(9)},x^{(10)},x^{(14)},x^{(15)},x^{(16)},x^{(17)},x^{(19)},x^{(20)},x^{(21)},x^{(22)},x^{(23)}\},$$ and $$C_2=\{
x^{(1)}, x^{(2)},x^{(4)},x^{(5)},x^{(8)},x^{(9)},x^{(10)},x^{(14)},x^{(15)},x^{(16)},x^{(17)},x^{(18)},x^{(19)}\}.$$
These cliques correspond to matrices in $\mathcal{Q}$ which are not permutation matrices. This means $G$ has exactly two generalized cospectral mates which can be recovered easily by the operation $Q^\T A(G) Q$.  For the sake of clarity, we write the associated matrices as follows:

Two matrices  corresponding to $C_1$ and $C_2$ are
\begin{equation}\small{
Q_1=\frac{1}{6}
\left(
\begin{array}{ccccccccccccc}
6 & 0 & 0 & 0 & 0 & 0 & 0 & 0 & 0 & 0 & 0 & 0 & 0 \\
0 & 0 & 0 & 0 & 0 & 0 & 0 & 0 & 0 & -3 & 3 & 3 & 3 \\
0 & 0 & 0 & 0 & 2 & 2 & 2 & 4 & -2 & -1 & -1 & -1 & 1 \\
0 & 6 & 0 & 0 & 0 & 0 & 0 & 0 & 0 & 0 & 0 & 0 & 0 \\
0 & 0 & 0 & 0 & 0 & 0 & 0 & 0 & 0 & 3 & -3 & 3 & 3 \\
0 & 0 & 0 & 0 & 4 & -2 & -2 & 2 & 2 & 1 & 1 & 1 & -1 \\
0 & 0 & 0 & 0 & 2 & 2 & 2 & -2 & -2 & 2 & 2 & 2 & -2 \\
0 & 0 & 0 & 0 & 0 & 0 & 0 & 0 & 0 & 3 & 3 & -3 & 3 \\
0 & 0 & 6 & 0 & 0 & 0 & 0 & 0 & 0 & 0 & 0 & 0 & 0 \\
0 & 0 & 0 & 6 & 0 & 0 & 0 & 0 & 0 & 0 & 0 & 0 & 0 \\
0 & 0 & 0 & 0 & 2 & 2 & 2 & -2 & 4 & -1 & -1 & -1 & 1 \\
0 & 0 & 0 & 0 & -2 & 4 & -2 & 2 & 2 & 1 & 1 & 1 & -1 \\
0 & 0 & 0 & 0 & -2 & -2 & 4 & 2 & 2 & 1 & 1 & 1 & -1 \\
\end{array}
\right)}
\end{equation}
and
\begin{equation}\small{
Q_2=\frac{1}{3}\left(
\begin{array}{ccccccccccccc}
3 & 0 & 0 & 0 & 0 & 0 & 0 & 0 & 0 & 0 & 0 & 0 & 0 \\
0 & 3 & 0 & 0 & 0 & 0 & 0 & 0 & 0 & 0 & 0 & 0 & 0 \\
0 & 0 & 0 & 0 & 0 & 0 & 0 & 1 & 1 & 1 & 2 & -1 & -1 \\
0 & 0 & 3 & 0 & 0 & 0 & 0 & 0 & 0 & 0 & 0 & 0 & 0 \\
0 & 0 & 0 & 3 & 0 & 0 & 0 & 0 & 0 & 0 & 0 & 0 & 0 \\
0 & 0 & 0 & 0 & 0 & 0 & 0 & 2 & -1 & -1 & 1 & 1 & 1 \\
0 & 0 & 0 & 0 & 0 & 0 & 0 & 1 & 1 & 1 & -1 & 2 & -1 \\
0 & 0 & 0 & 0 & 3 & 0 & 0 & 0 & 0 & 0 & 0 & 0 & 0 \\
0 & 0 & 0 & 0 & 0 & 3 & 0 & 0 & 0 & 0 & 0 & 0 & 0 \\
0 & 0 & 0 & 0 & 0 & 0 & 3 & 0 & 0 & 0 & 0 & 0 & 0 \\
0 & 0 & 0 & 0 & 0 & 0 & 0 & 1 & 1 & 1 & -1 & -1 & 2 \\
0 & 0 & 0 & 0 & 0 & 0 & 0 & -1 & 2 & -1 & 1 & 1 & 1 \\
0 & 0 & 0 & 0 & 0 & 0 & 0 & -1 & -1 & 2 & 1 & 1 & 1 \\
\end{array}
\right)}
\end{equation}
The adjacency matrices of the corresponding generalized cospectral mates are
\begin{equation}\small{
	Q_1^\T AQ_1=\left(
	\begin{array}{ccccccccccccc}
	0 & 0 & 0 & 0 & 1 & 0 & 0 & 0 & 1 & 0 & 0 & 1 & 1 \\
	0 & 0 & 0 & 0 & 1 & 1 & 1 & 0 & 0 & 0 & 0 & 1 & 1 \\
	0 & 0 & 0 & 0 & 1 & 0 & 0 & 1 & 0 & 0 & 0 & 0 & 0 \\
	0 & 0 & 0 & 0 & 1 & 0 & 0 & 0 & 1 & 0 & 1 & 0 & 1 \\
	1 & 1 & 1 & 1 & 0 & 0 & 1 & 0 & 0 & 1 & 0 & 0 & 1 \\
	0 & 1 & 0 & 0 & 0 & 0 & 1 & 0 & 0 & 0 & 0 & 1 & 1 \\
	0 & 1 & 0 & 0 & 1 & 1 & 0 & 0 & 1 & 0 & 0 & 1 & 1 \\
	0 & 0 & 1 & 0 & 0 & 0 & 0 & 0 & 1 & 0 & 1 & 0 & 1 \\
	1 & 0 & 0 & 1 & 0 & 0 & 1 & 1 & 0 & 0 & 1 & 1 & 0 \\
	0 & 0 & 0 & 0 & 1 & 0 & 0 & 0 & 0 & 0 & 0 & 0 & 0 \\
	0 & 0 & 0 & 1 & 0 & 0 & 0 & 1 & 1 & 0 & 0 & 0 & 0 \\
	1 & 1 & 0 & 0 & 0 & 1 & 1 & 0 & 1 & 0 & 0 & 0 & 0 \\
	1 & 1 & 0 & 1 & 1 & 1 & 1 & 1 & 0 & 0 & 0 & 0 & 0 \\
	\end{array}
	\right).}
\end{equation}
and
\begin{equation}\small{Q_2^\T AQ_2=
\left(
\begin{array}{ccccccccccccc}
0 & 1 & 0 & 1 & 0 & 0 & 0 & 1 & 0 & 0 & 0 & 0 & 1 \\
1 & 0 & 1 & 0 & 0 & 0 & 1 & 0 & 1 & 1 & 1 & 0 & 1 \\
0 & 1 & 0 & 1 & 0 & 0 & 0 & 1 & 1 & 1 & 0 & 0 & 0 \\
1 & 0 & 1 & 0 & 0 & 0 & 0 & 1 & 1 & 1 & 0 & 0 & 0 \\
0 & 0 & 0 & 0 & 0 & 0 & 1 & 1 & 0 & 0 & 1 & 0 & 0 \\
0 & 0 & 0 & 0 & 0 & 0 & 0 & 1 & 0 & 0 & 1 & 0 & 0 \\
0 & 1 & 0 & 0 & 1 & 0 & 0 & 1 & 0 & 0 & 0 & 0 & 1 \\
1 & 0 & 1 & 1 & 1 & 1 & 1 & 0 & 0 & 1 & 0 & 0 & 0 \\
0 & 1 & 1 & 1 & 0 & 0 & 0 & 0 & 0 & 1 & 0 & 0 & 0 \\
0 & 1 & 1 & 1 & 0 & 0 & 0 & 1 & 1 & 0 & 0 & 0 & 1 \\
0 & 1 & 0 & 0 & 1 & 1 & 0 & 0 & 0 & 0 & 0 & 0 & 1 \\
0 & 0 & 0 & 0 & 0 & 0 & 0 & 0 & 0 & 0 & 0 & 0 & 1 \\
1 & 1 & 0 & 0 & 0 & 0 & 1 & 0 & 0 & 1 & 1 & 1 & 0 \\
\end{array}
\right).}
\end{equation}
\subsection{Extending to almost controllable graphs}	

Now we assume the given graph $G$ is almost controllable. Lemma  \ref{nontrsol} can be easily extended to almost controllable graphs by a natural modification of Eq.~\eqref{sys1}: replacing $W^{(p_i)}(G)$ by $W^{(p_i)}_0(G)$. Accordingly, we define $\Omega(G)$ similarly as in Definition \ref{OG}. To obtain a correct version of Theorem \ref{nsc2} for almost controllable graphs, we need to distinguish two classes according to their automorphism groups.

A graph is \emph{asymmetric} if its automorphism group is trivial; and is  \emph{symmetric} otherwise.
\begin{lemma}\label{sva}
	Let $G$ and $H$ be two almost controllable graphs having the same generalized spectrum.  Let $Q_1$ and $Q_2$ be the two rational regular orthogonal matrices satisfying $Q^\T A(G)Q=A(H)$. Then $Q_1$ and $Q_2$ are column permutation equivalent if and only if $H$ is symmetric.
\end{lemma}
\begin{proof}
	Suppose that $Q_1$ and $Q_2$ are column permutation equivalent. Then there exists a permutation matrix $P$ such that $Q_2=Q_1P$. Since $Q_2\neq Q_1$, we see that $P$ is not the identity matrix. Since both $Q_1$ and $Q_2$ satisfy the equation $Q^\T A(G)Q=A(H)$, we have $$A(H)=Q_2^\T A(G)Q_2=(Q_1P)^\T A(G)(Q_1P)=P^\T(Q_1^\T A(G)Q_1)P=P^\T A(H)P.$$
	This means that $H$ is symmetric.
	
	Conversely, if there exists a permutation matrix $P$ other than the identity matrix such that $P^\T A(H)P=A(H)$, then it is easy to see that $\hat{Q}=Q_1P$ also satisfies the equation  $Q^\T A(G)Q=A(H)$. Clearly, $\hat{Q}$ is a rational regular orthogonal matrix and is different from $Q_1$. Thus, $\hat{Q}=Q_2$, that is, $Q_2$ and $Q_1$ are column permutation equivalent. This completes the proof.
\end{proof}
Let $K=\{\xi_1,\ldots,\xi_n\}$ be a maximum clique of $\Omega(G)$. We define $\phi(K)$ to be the isomorphic class of the graph with adjacency matrix $Q^\T A(G)Q$, where $Q=\frac{1}{L}[\xi_1,\ldots,\xi_n]$. We regard $\phi(K)$ as an unlabeled graph and then $\phi$ gives a map from all maximum cliques to $\mathcal{C}(G)$, the set of all unlabeled graphs generalized cospectral with $G$.

Using Lemma \ref{sva}, we can obtain the following analogy of Theorem \ref{nsc2}.

\begin{theorem}
	Let $G$ be almost controllable. Then the map $\phi$ is surjective. Moreover, any symmetric graph in $\mathcal{C}(G)$ has exactly one preimage whereas any asymmetric graph in  $\mathcal{C}(G)$ has exactly two preimages.
\end{theorem}
\begin{corollary}\label{dgsacon}
	Let $G$ be a symmetric (resp. asymmetric) almost controllable graph. Then $G$ has a generalized cospectral mate if and only if $\Omega(G)$ has at least two (resp. three) cliques of order $n$.
\end{corollary}

\begin{example}\normalfont{
Let $G$ be the graph with $n=9$ vertices whose adjacency matrix is
\begin{equation}A=\small{\left(
	\begin{array}{ccccccccc}
	0 & 1 & 1 & 1 & 1 & 0 & 1 & 1 & 1 \\
	1 & 0 & 1 & 0 & 1 & 1 & 0 & 0 & 0 \\
	1 & 1 & 0 & 1 & 0 & 1 & 1 & 0 & 0 \\
	1 & 0 & 1 & 0 & 1 & 0 & 0 & 1 & 0 \\
	1 & 1 & 0 & 1 & 0 & 1 & 0 & 1 & 0 \\
	0 & 1 & 1 & 0 & 1 & 0 & 1 & 0 & 0 \\
	1 & 0 & 1 & 0 & 0 & 1 & 0 & 1 & 0 \\
	1 & 0 & 0 & 1 & 1 & 0 & 1 & 0 & 1 \\
	1 & 0 & 0 & 0 & 0 & 0 & 0 & 1 & 0 \\
	\end{array}
	\right).}
\end{equation}
The graph $G$ is almost controllable and asymmetric. By a similar procedure as we do for Example \ref{g13} in the last subsection, we obtain $L=128$ and graph $\Omega(G)$. Table \ref{vexem2} records the vertex set of $\Omega(G)$.
\begin{table}
	\centering
	 \caption{Vertex set of $\Omega(G)$.}
	 \label{vexem2}
	 	\small{
\begin{tabular}{cc|cc|cc}
	\toprule
	$i$ & $\frac{1}{16}v_i$ & $i$ & $\frac{1}{16}v_i$ & $i$ & $\frac{1}{16}v_i$\\
	\midrule
1&$({8,0,0,0,0,0,0,0,0})$&14&$({6,-2,-2,0,2,2,2,2,-2})$&27&$({0,4,4,0,0,0,0,-4,4})$\\
2&$({0,8,0,0,0,0,0,0,0})$&15&$({-2,6,-2,0,2,2,2,2,-2})$&28&$({0,4,4,0,0,0,0,4,-4})$\\
3&$({0,0,8,0,0,0,0,0,0})$&16&$({-2,-2,6,0,2,2,2,2,-2})$&29&$({-2,-2,-2,0,2,6,2,2,2})$\\
4&$({0,0,0,8,0,0,0,0,0})$&17&$({-2,-2,-2,0,2,2,2,2,6})$&30&$({2,2,2,0,-2,2,-2,-2,6})$\\
5&$({0,0,0,0,8,0,0,0,0})$&18&$({4,-2,-2,4,4,0,0,-2,2})$&31&$({-2,4,4,4,2,-2,-2,0,0})$ \\
6&$({0,0,0,0,0,8,0,0,0})$&19&$({0,2,2,4,0,-4,4,2,-2})$&32&$({2,0,0,4,-2,2,2,4,-4})$\\
7&$({0,0,0,0,0,0,8,0,0})$&20&$({0,2,2,4,0,4,-4,2,-2})$&33&$({-4,2,2,4,4,0,0,2,-2})$\\
8&$({0,0,0,0,0,0,0,8,0})$&21&$({-4,0,0,0,4,4,4,0,0})$&34&$({4,2,2,-4,4,0,0,2,-2})$\\
9&$({0,0,0,0,0,0,0,0,8})$&22&$({4,0,0,0,-4,4,4,0,0})$&35&$({4,2,2,4,-4,0,0,2,-2})$\\
10&$({2,2,2,0,6,-2,-2,-2,2})$&23&$({4,0,0,0,4,-4,4,0,0})$&36&$({-1,-1,-1,4,1,3,-3,5,1})$\\
11&$({2,2,2,0,-2,6,-2,-2,2})$&24&$({4,0,0,0,4,4,-4,0,0})$&37&$({1,1,1,4,-1,1,-5,3,3})$\\
	12 & $({2, 2, 2, 0, -2, -2, 6, -2, 2})$&25&$({0, -4, 4, 0, 0, 0, 0, 4, 4})$& & \\
	13 & $({2, 2, 2, 0, -2, -2, -2, 6, 2})$& 26 &$( {0, 4, -4, 0, 0, 0, 0, 4, 4} )$  &  & \\
					
	\bottomrule
\end{tabular}
}
\end{table}

The graph $\Omega(G)$ has exactly 8 maximum cliques. Finally, using these cliques, we completely determine the set $\mathcal{C}(G)$. The exact map from maximum cliques to $\mathcal{C}(G)$ is illustrated in Table \ref{corr}.
\begin{table}[h]
	\centering
		\caption{Correspondence between maximum cliques and $\mathcal{C}(G)$}\label{corr}
	\begin{tabular}{ c  l  }
		\toprule
		$\mathcal{C}(G)$ & Maximum cliques in $\Omega(G)$ \\
		\midrule
		\begin{minipage}[b]{0.15\columnwidth}
			\centering
			\raisebox{-.5\height}{\includegraphics[width=\linewidth]{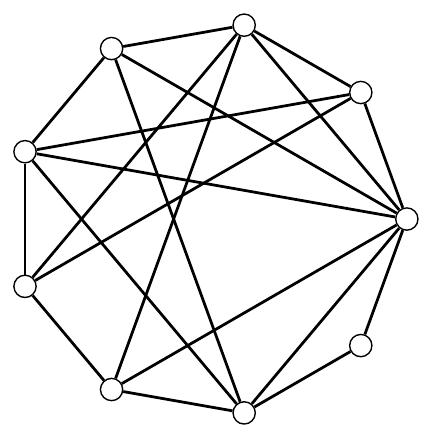}}
		\end{minipage}
		& $\{v_1,v_2,v_3,v_4,v_5,v_6,v_7,v_8,v_9\}$, $\{v_4, v_{10},v_{11}, v_{12}, v_{13}, v_{14}, v_{15}, v_{16}, v_{17}\}$\\
			\begin{minipage}[b]{0.15\columnwidth}
			\centering
			\raisebox{-.5\height}{\includegraphics[width=\linewidth]{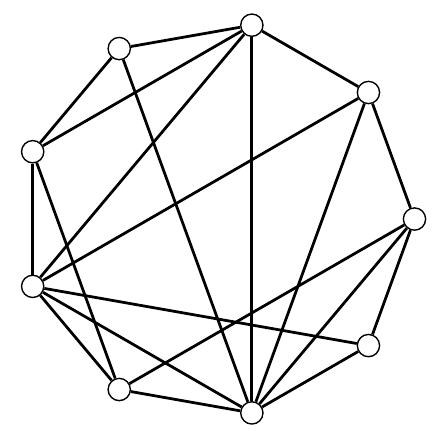}}
		\end{minipage}
		& $\{v_6, v_7, v_{18}, v_{25},v_{26}, v_{27}, v_{33}, v_{34}, v_{35}\}$, $\{v_{11}, v_{12}, v_{18}, v_{21}, v_{25}, v_{26}, v_{31}, v_{32}, v_{34}\}$\\
			\begin{minipage}[b]{0.15\columnwidth}
			\centering
			\raisebox{-.5\height}{\includegraphics[width=\linewidth]{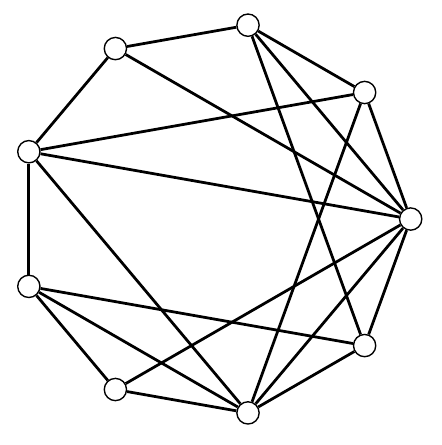}}
		\end{minipage}
	& $\{v_1, v_4, v_5, v_6, v_7, v_{25}, v_{26}, v_{27}, v_{28}\}$, $\{v_4, v_{10}, v_{11}, v_{12}, v_{14}, v_{21}, v_{25}, v_{26}, v_{28}\}$\\
		\begin{minipage}[b]{0.15\columnwidth}
		\centering
		\raisebox{-.5\height}{\includegraphics[width=\linewidth]{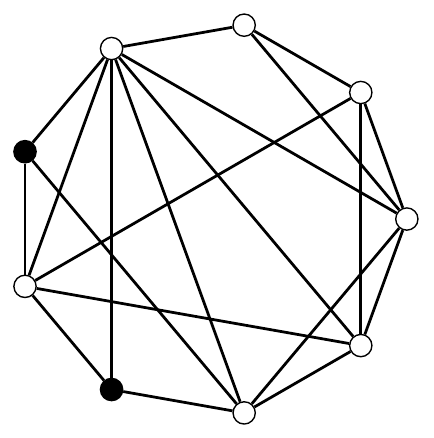}}
	\end{minipage}
	&$\{v_{18}, v_{19}, v_{20}, v_{21}, v_{22}, v_{25}, v_{26}, v_{27}, v_{34}\}$\\
		\begin{minipage}[b]{0.15\columnwidth}
		\centering
		\raisebox{-.5\height}{\includegraphics[width=\linewidth]{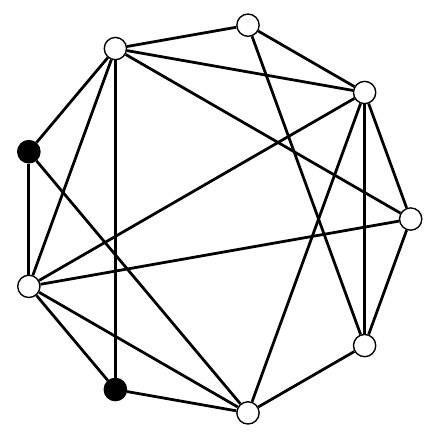}}
	\end{minipage}
	&$\{v_{4}, v_{21}, v_{22}, v_{23}, v_{24}, v_{25}, v_{26}, v_{27}, v_{28}\}$\\
		 \bottomrule
	\end{tabular}
\begin{tablenotes}
	\footnotesize
	\item*The last two graphs are symmetric; swapping the two vertices represented by solid circles (while fixing the remaining vertices) gives the desired automorphism for either graph.
\end{tablenotes}

\end{table}
}
\end{example}
\section{Simulation on Haemers' conjecture}\label{sim}
As an application, we give some simulation results on Haemers' conjecture that almost all graphs are DGS. For each $n$, we randomly generate 10,000 graphs that are controllable or almost controllable. We use the Mathematica command RandomGraph[BernoulliGraphDistri-bution[$n$, 0.5]]  to generate graphs and choose those controllable or almost controllable ones. For each generated graph, we use Corollary \ref{dgscon} or Corollary \ref{dgsacon} to check whether it is DGS. For convenience, we summarize the main procedure in Algorithm \ref{A}, where we only consider controllable graphs for simplicity.

In our real implementation, we also dismiss some graphs for which the computations of $\Omega(G)$ are complex. We set $2^{16}$ as the upper bound  in our program. In Statement 10 of Algorithm \ref{A}, we compute the complexity for the first two steps. The complexity of Step 1 is  measured by the number of solutions of the equation $(W^{(p_i)}(G))^\T \xi_i\equiv 0\pmod {p_i^{t_i}}$, which can be obtained by Corollary \ref{numsl}. The complexity of Step 2 is measured by the cardinality of $\Pi_1\times \Pi_2\times\cdots\times  \Pi_s$, which is easy to obtain in the execution. We do not set obstacles using the upper bound in Step 3 but use Corollary \ref{necforperfect} to rule out solutions (from Step 2) that can not be $m$-perfect.
\begin{algorithm}[t]
	\caption{Determining whether a controllable graph is DGS}	\label{A}
	\hspace*{0.02in} {\bf Input:} 
	an $n$-vertex controllable graph $G$.\\
	\hspace*{0.02in} {\bf Output:} 
	DGS or non-DGS
	\begin{algorithmic}[1]
		\State Compute $d:=\gcd(\Delta(G),d_n(W(G)))$.   \;		
		\State Factor $d$ and set $S:=\{2\}\cup\{p\colon\, p$ is an odd prime factor of $d$ with $p^2\mid \Delta(G)\} $.\;
		\State Set $s:=0$.
		\For{ each $p$ in $S$} 
		\If{$\ord_{p}(d_n(W^{(p)}(G)))\ge 1$}
		\State	Update $s:=s+1$.\;
		\State Set $p_s:=p$ and $t_s:=\ord_{p}(d_n(W^{(p)}(G)))$.\;
		\EndIf
		\EndFor
		\If{$s=0$} \Return DGS.\EndIf
		\State Set $L:=\prod_{i=1}^s p_i^{t_i}$.
		\State Solve system \eqref{sys1} using the three steps described in Sec.~\ref{threesteps}.
		\State Construct the graph $\Omega(G)$ as in Definition \ref{OG}.
		\If{ $\Omega(G)$ has exactly one maximum clique} \Return DGS.\EndIf
		\State\Return non-DGS.
	\end{algorithmic}
\end{algorithm}

Table \ref{res} records the experimental result in one implementation. The Mathematica codes  can be found at https://github.com/wangweiAHPU/HaemersConjecture. For example, when $n=10$, among the 10,000 graphs which are controllable or almost controllable, 7977 graphs are guaranteed to be DGS while 1875 graphs are guaranteed to be not DGS. The remaining 148 graphs are undecided as they reach the upper bound on the complexity in the computation of $\Omega(G)$. Among the 7977 DGS graphs, 5687 graphs  have $K_n$ as  $\Omega(G)$. As $n$ increases, more than $99\%$ of the graphs are guaranteed to be DGS. Furthermore, the vast majority of these DGS graphs have $K_n$ as $\Omega(G)$ which trivially has exactly one maximum clique. This phenomenon  gives a simple explanation of why our proposed algorithm runs fast in practice.
	\begin{table}
	\centering
	\caption{Fractions of DGS graphs}	\label{res}
	
	\begin{tabular}{ccccc}
		\toprule
		$n$ & \# DGS graphs & \# graphs with $\Omega(G)=K_n$&\# non-DGS graphs & \# Undecided\\
		\midrule
		7 &9755&7869  & 245 & 0\\
		8 &9021&6817  & 970 & 9\\
		9&8305&5914&1626&69\\
		10&7977&5687&1875&148\\
		11& 7984&5696& 1812 &204 \\
		12& 8055&5802& 1702 & 243\\
		13& 8378 & 6162&1405  &217 \\
		14& 8724 & 6735 &1097  &179\\
		15&9108& 7321 &743  & 149\\
		16&9373 &7701&522  &105\\
		17& 9519&8176& 397 & 84 \\
		18	& 9655 &8573 & 264 &81\\
		19	&9762 &8920 &171  &67 \\
		20	& 9839 &9232 & 101 &60\\
		21  &9883 &9465 &53  &64\\
		22 & 9903& 9609& 36&61\\
		23 &9908&9696&23&69\\
		24 & 9928&9778&18 & 54\\
		25&9909&9831&5&86\\
		26& 9944&9886&1 & 55\\
		27&9934&9896&1&65\\
		28 &9940&9925 &1 &59 \\
		29&9942&9928&0&58\\
		30 &9952&9943 &0 &48 \\
		35 &9931 &9930& 0&69 \\
		40 &9949 &9949& 0& 51\\
		50 & 9945&9945& 0&55 \\
		\bottomrule
	\end{tabular}
\end{table}

We believe that, compared with Table \ref{computer}, the experimental results give us more confidence in Haemers' conjecture. We end this paper with the following conjecture, which is naturally inspired by our experiment.
\begin{conjecture}
	Almost all controllable graphs $G$ have complete graphs as $\Omega(G)$.
\end{conjecture}

\section*{Acknowledgments}
This work is supported by the	National Natural Science Foundation of China (Grant Nos. 12001006, 11971376 and 11971406) and the Scientific Research Foundation of Anhui Polytechnic University (Grant No.\,2019YQQ024).

\end{document}